\documentclass[12pt]{article}
\usepackage{amsmath,amssymb,graphicx,amsthm}

\newcommand{\dueto}[1]{\textup{\textbf{(#1) }}}
\newcommand{\tmem}[1]{{\em #1\/}}
\newcommand{\tmmathbf}[1]{\ensuremath{\boldsymbol{#1}}}
\newcommand{\tmop}[1]{\ensuremath{\operatorname{#1}}}
\newcommand{\tmtextit}[1]{{\itshape{#1}}}
\newcommand{\udots}{{\mathinner{\mskip1mu\raise1pt\vbox{\kern7pt\hbox{.}}\mskip2mu\raise4pt\hbox{.}\mskip2mu\raise7pt\hbox{.}\mskip1mu}}}
\newtheorem{lemma}{Lemma}
\newtheorem{corollary}{Corollary}
\newtheorem{proposition}{Proposition}
\newtheorem{theorem}{Theorem}

\begin{document}

\title{Factorial Schur Functions and the Yang-Baxter Equation}\author{Daniel
Bump, Peter J.\ McNamara and Maki Nakasuji}\maketitle

\bigbreak\noindent
{\footnotesize \textbf{Abstract}. 
Factorial Schur functions are generalizations of Schur functions
that have, in addition to the usual variables, a second family
of ``shift'' parameters. We show that a factorial Schur function
times a deformation of the Weyl denominator may be expressed as
the partition function of a particular statistical-mechanical
system (six-vertex model).  The proof is based on the
Yang-Baxter equation. There is a deformation parameter $t$ which
may be specialized in different ways. If $t=-1$, then we recover
the expression of the factorial Schur function as a ratio of
alternating polynomials. If $t=0$, we recover the description as
a sum over tableaux. If $t=\infty$ we recover a description of
Lascoux that was previously considered by the second author. We also are
able to prove using the Yang-Baxter equation the asymptotic
symmetry of the factorial Schur functions in the shift
parameters. Finally, we give a proof using our methods of the
dual Cauchy identity for factorial Schur functions.  Thus using
our methods we are able to give thematic proofs of many of the
properties of factorial Schur functions.}

\bigbreak\noindent
\textit{Dedicated to Professor Fumihiro Sato}

\section{Introduction}
{\tmem{Factorial Schur functions}} are generalizations of ordinary Schur
functions $s_\lambda(z)=s_{\lambda} (z_1, \cdots, z_n)$ for which a surprising
amount of the classical theory remains valid. In addition to the usual spectral
parameters $z=(z_1,\cdots,z_n)$ and the partition $\lambda$ they involve a set $\alpha =
(\alpha_1, \alpha_2, \alpha_3, \cdots)$ of shifts that can be arbitrary
complex numbers (or formal variables), and are denoted $s_\lambda(z|\alpha)$. In
the original paper of Biedenharn and Louck~{\cite{BiedenharnLouck}}, only
the special case where $\alpha_n = 1 - n$ was considered. Their
motivation, inspired by questions from mathematical physics, was to the
decomposition of tensor products of representations with using particular
bases. It turns out that factorial Schur functions are the same as double
Schubert polynomials for Grassmannian permutations, and in this form they
appeared even earlier in Lascoux and Sch\"utzenberger~\cite{LascouxSchutz},
whose motivation (from algebraic geometry) was completely different.

Other early foundational papers are Chen and Louck~{\cite{ChenLouck}}, who
gave new foundations based on divided difference operators, and Goulden and
Hamel~{\cite{GouldenHamel}} where the analogy between Schur functions and
factorial Schur functions was further developed. In particular they gave a
Jacobi-Trudi identity. See Louck~{\cite{Louck}} for further historical
remarks.

Biedenharn and Louck (in the special case $\alpha_n=1-n$) defined $s_{\lambda}
(z| \alpha)$ to be a sum over Gelfand-Tsetlin patterns, and this definition
extends to the general case. Translated into the equivalent language of
tableaux, their definition is equivalent to (\ref{tableaudef}) below. It was
noticed independently by Macdonald~{\cite{MacdonaldVariations}} and by Goulden
and Greene~{\cite{GouldenGreen}} that one could generalize the factorial Schur
functions of Biedenharn and Louck by making use of an arbitrary set $\alpha$
of shifts. Macdonald observed an alternative definition of the factorial Schur
functions as a ratio of two alternating polynomials, generalizing the Weyl
character formula. This definition is (\ref{fsfdef}) below.

Both Macdonald and Goulden and Greene also noticed a relationship with what are
called {\tmem{supersymmetric Schur functions}}. These are symmetric functions in
two sets of variables, $z = (z_1, z_2, \cdots)$ and $w = (w_1, w_2, \cdots)$.
They are defined in terms of the ordinary Schur functions by
\[ s_{\lambda} (z\|w) = \sum_{\mu, \nu} c^{\lambda}_{\mu \nu} \, s_{\mu} (z)
   s_{\nu'} (w), \]
where $c_{\mu \nu}^{\lambda}$ is the Littlewood-Richardson coefficient, 
$\mu$ and $\nu$ run through partitions and
$\nu'$ is the conjugate partition. The relationship between the factorial
Schur functions and the supersymmetric Schur functions is this: although the
$s_{\lambda} (z| \alpha)$ are symmetric in the $z_i$ they are not symmetric in
the $\alpha_i$. Nevertheless, as the number $n$ of the parameters
$z_i$ tends to infinity, they become symmetric in the $\alpha_i$ in a certain
precise sense, and in the limit, they stabilize. Thus in a suitable sense
\begin{equation}
  \label{sslimit} \lim_{n \longrightarrow \infty} s_{\lambda} (z| \alpha) =
  s_{\lambda} (z\| \alpha) .
\end{equation}

Another important variant of the factorial Schur functions are the
{\tmem{shifted Schur functions}} that were proposed by Olshanskii, and
developed by Okounkov and Olshanskii~{\cite{OO}},~{\cite{OOII}}. Denoted
$s_{\lambda}^{\ast} (x_1, \cdots, x_n)$, they are essentially the same as the
factorial Schur functions of Biedenharn and Louck, but incorporate shifts in
the parameters so that they are no longer symmetric in the usual sense, but at
least satisfy the stability property $s_{\lambda}^{\ast} (x_1, \cdots, x_n) =
s_{\lambda}^{\ast} (x_1, \cdots, x_n, 0)$. These were applied to the
representation theory of the infinite symmetric group.

Molev and Sagan~{\cite{MolevSagan}} give various useful results for factorial
Schur functions, including a Littlewood-Richardson rule. A further
Littlewood-Richardson rule was found by Kreiman~{\cite{Kreiman}}. Knutson and
Tao~{\cite{KnutsonTao}} show that factorial Schur functions correspond to
Schubert classes in the equivariant cohomology of Grassmanians. See also
Mihalcea~\cite{Mihalcea} and Ikeda and Naruse~\cite{IkedaNaruse}.

Tokuyama~{\cite{Tokuyama}} gave a formula for Schur functions that depends on
a parameter $t$. This formula may be regarded as a deformation of the Weyl
character formula. It was shown by Hamel and King~{\cite{HamelKing}} that
Tokuyama's formula could be generalized and reformulated as the evaluation of
the partition function for a statistical system based on the six-vertex model
in the free-fermionic regime. Brubaker, Bump and Friedberg~{\cite{hkice}} gave
further generalizations of the results of Hamel and King, with new proofs
based on the Yang-Baxter equation. More specifically, they used the fact
that the six-vertex model in the free-fermionic regime satisfies a
parametrized Yang-Baxter equation with nonabelian parameter group $\tmop{GL}
(2) \times \tmop{GL} (1)$ to give statistical-mechanical systems whose
partition functions were Schur functions times a deformation of the Weyl
denominator. This result generalizes the results of Tokuyama and of Hamel and
King.

Our main new result (Theorem~\ref{icerepn}) is a Tokuyama-like formula for
factorial Schur functions. This is a simultaneous generalization
of~{\cite{hkice}} and of the representations of Lascoux~\cite{LascouxSchubert}
and of McNamara~\cite{McNamara}. As in~{\cite{hkice}} we will consider
partition functions of statistical-mechanical systems in the free-fermionic
regime. A significant difference between this paper and that was that
in~{\cite{hkice}} the Boltzmann weights were constant along the rows,
depending mainly on the choice of a parameter $z_i$. Now we will
consider systems in which we assign a parameter $z_i$ to each row,
but also a shift parameter $\alpha_j$ to each column. Furthermore, we will
make use of a deformation parameter $t$ that applies to the entire system.
We will show that the partition function may be expressed as the product of a
factor (depending on $t$) that may be recognized as a deformation of the Weyl
denominator, times the factorial Schur function $s_{\lambda} (z| \alpha)$.

The proof of Theorem~\ref{icerepn} depends on the Yang-Baxter equation.
We feel that it is significant that the Yang-Baxter equation can be
made a central tool in the theory of factorial Schur functions. The
results in the paper after Theorem~\ref{icerepn} are mainly already
known, but we will reprove them using our methods---either deducing
them from Theorem~\ref{icerepn} or giving proofs using the same tool
(free-fermionic Yang-Baxter equation).

By specializing $t$ in the formula of Theorem~\ref{icerepn}, we will obtain
different formulas for the factorial Schur functions. Taking $t = - 1$, we
obtain the representation as a ratio of alternating polynomials, which was
Macdonald's generalization of the Weyl character formula. This is the
formula we take as the definition of the factorial Schur functions, though
other definitions are possible.

There are two specializations $t$ in which the Weyl denominator in
Theorem~\ref{icerepn} reduces to a monomial. Taking $t = 0$, we obtain the
tableau definition of the factorial Schur functions. When $t = \infty$ we
recover another representation of the factorial Schur functions. Indeed
Lascoux~{\cite{LascouxSchubert}} found six-vertex model representations of
Grassmannian Schubert polynomials. A proof of this representation based on the
Yang-Baxter equation was subsequently found by McNamara~{\cite{McNamara}}. It
is this representation that we obtain when $t=\infty$.

Although we do not prove the supersymmetric limit
(\ref{sslimit}), we will at least prove the key fact that the $s_{\lambda} (z|
\alpha)$ are asymptotically symmetric in the $\alpha_j$ as the number $n$ of
parameters $z_i$ tends to infinity. We will obtain this by another
application of the Yang-Baxter equation.  We also give a proof of the dual
Cauchy identity for factorial Schur functions using our methods.

In addition to~\cite{LascouxSchubert} and~\cite{McNamara},
Zinn-Justin~\cite{ZJ}, \cite{ZJ1} gave another interpretation of factorial
Schur functions as transition matrices for a lattice model that may be
translated into a free-fermionic five-vertex model. It is unclear whether
Zinn-Justin's representation may also be obtained from Theorem~\ref{icerepn}
ours by specialization, but it is certainly very similar.

We would like to thank H.~Naruse for helpful comments on this paper and the
referee for careful reading. This work was supported in part by JSPS Research
Fellowship for Young Scientists and by NSF grants DMS-0652817 and DMS-1001079.

\section{Yang-Baxter equation}\label{ybesection}
We review the six-vertex model and a case of the Yang-Baxter equation
from~{\cite{hkice}}. We will consider a planar graph. Each vertex $v$ is
assumed to have exactly four edges adjacent to it. \textit{Interior edges}
adjoin two vertices, and around the boundary of the graph we allow
\textit{boundary edges} that adjoin only a single vertex. Every vertex has six
numbers $a_1 (v), a_2 (v), b_1 (v), b_2 (v), c_1 (v), c_2 (v)$ assigned to
it. These are called the {\tmem{Boltzmann weights}} at the vertex. By
a \textit{spin} we mean an element of the two-element set $\{+,-\}$.
In addition to the graph, the Boltzmann weights at each vertex,
we will also assign a spin to each boundary edge. Once we
have specified the graph, the Boltzmann weights at the vertices, and the
boundary spins, we have specified a statistical system~$\mathfrak{S}$.

A {\tmem{state}} $\mathfrak{s}$ of the system will be an assignment of spins
to the interior edges. Given a state of the system, every edge, boundary or
interior, has a spin assigned to it. Then every vertex will have a
definite configuration of spins on its four adjacent edges, and we
assume these to be in one of the two orientations listed in (\ref{boltzmannweights}).
Then let $\beta_{\mathfrak{s}}(v)$ equal
$a_1 (v), a_2 (v), b_1 (v), b_2 (v),
c_1 (v)$ or $c_2 (v)$ depending on the configuration of spins on the adjacent
edges. If $v$ does not appear in the table, the weight is zero.

\begin{equation}\label{boltzmannweights} \begin{array}{|c|c|c|c|c|c|}
     \hline
     \includegraphics{weighta1.mps} & \includegraphics{weighta2.mps} &
     \includegraphics{weightb1.mps} & \includegraphics{weightb2.mps} &
     \includegraphics{weightc1.mps} & \includegraphics{weightc2.mps}\\
     \hline
     \includegraphics{rota1.mps} & \includegraphics{rota2.mps} &
     \includegraphics{rotb1.mps} & \includegraphics{rotb2.mps} &
     \includegraphics{rotc1.mps} & \includegraphics{rotc2.mps}\\
     \hline
     a_1 (v) & a_2 (v) & b_1 (v) & b_2 (v) & c_1 (v) & c_2 (v)\\
     \hline
   \end{array} \end{equation}

The \textit{Boltzmann weight of the state $\beta (\mathfrak{s})$} is the product
$\prod_v \beta_{\mathfrak{s}} (v)$ of the Boltzmann weights at every vertex.
We only need to consider configurations in which the spins adjacent to each
vertex are in one of the configurations from the above table; if this is true,
the state is called {\tmem{admissible}}. A state that is not admissible has
Boltzmann weight zero.

The {\tmem{partition function $Z (\mathfrak{S})$}} is $\sum_{\mathfrak{s}}
\beta (\mathfrak{s})$, the sum of the Boltzmann weights of the states. We may
either include or exclude the inadmissible states from this sum, since they
have Boltzmann weight zero.

If at the vertex $v$ we have
\[ a_1 (v) a_2 (v) + b_1 (v) b_2 (v) - c_1 (v) c_2 (v) = 0, \]
the vertex is called {\tmem{free-fermionic}}. We will only consider systems
that are free-fermionic at every vertex.

Korepin, Boguliubov and Izergin~\cite{KBI} describe a nonabelian parametrized
Yang-Baxter equation for the free-fermionic six-vertex model with parameter
group $\Gamma=SL(2,\mathbb{C})$. Concretely this means that there is a
map $R:\Gamma\to\tmop{End}(V\otimes V)$, where $V$ is a two-dimensional
vector space, such that if $\gamma,\delta\in\Gamma$ then
\[R(\gamma)_{12}R(\gamma\delta)_{13}R(\delta)_{23}=
R(\delta)_{23}R(\gamma\delta)_{13}R(\gamma)_{12},\]
where if $R\in\tmop{End}(V\otimes V)$ then $R_{ij}$ means $R\times I_V$ acting
on $V\otimes V\otimes V$ with $R$ acting on the $i,j$ tensor components,
and the identity on the remaining component. Scalar matrices can obviously
be added to $\Gamma$ so their actual group is $SL(2,\mathbb{C})\times\mathbb{C}^\times$.
A statement with a slightly larger parameter group
$GL(2,\mathbb{C})\times\mathbb{C}^\times$ is in Brubaker, Bump and
Friedberg~\cite{hkice}. The nonzero components of $R(\gamma)$ if
written with respect to a standard basis of $V\otimes V$ will be
the Boltzmann weights of a free-fermionic vertex. This has the following
explicit reformulation:
\begin{proposition}
  \textbf{\cite[Theorem 3]{hkice}}
  \label{ybe}Let $v, w$ be vertices with free-fermionic Boltzmann weights.
  Define another type of vertex $u$ with
  \begin{eqnarray*}
    a_1 (u) & = & a_1 (v) a_2 (w) + b_2 (v) b_1 (w),\\
    a_2 (u) & = & b_1 (v) b_2 (w) + a_2 (v) a_1 (w),\\
    b_1 (u) & = & b_1 (v) a_2 (w) - a_2 (v) b_1 (w),\\
    b_2 (u) & = & - a_1 (v) b_2 (w) + b_2 (v) a_1 (w),\\
    c_1 (u) & = & c_1 (v) c_2 (w),\\
    c_2 (u) & = & c_2 (v) c_1 (w) .
  \end{eqnarray*}
  Then for any assignment of edge spins $\varepsilon_i \in \{\pm\}$ ($i = 1,
  2, 3, 4, 5, 6$) the following two configurations have the same partition
  function:
  
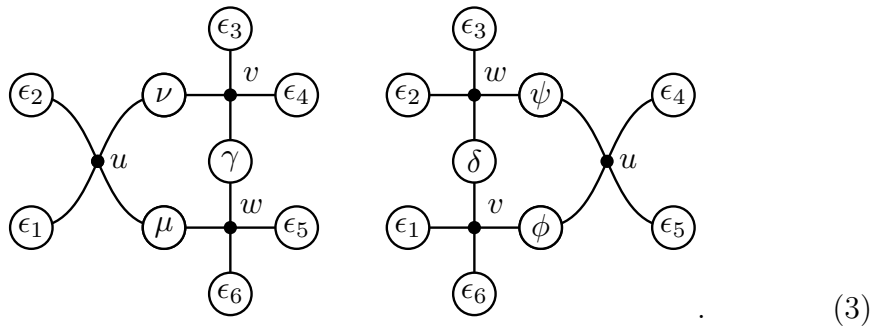
\begin{equation}
    \includegraphics{ybl.mps} \qquad \includegraphics{ybr.mps}.
  \end{equation}
\end{proposition}

Note that by the definition of the partition function, the interior edge spins
(labeled $\nu, \gamma, \mu$ and $\delta, \psi, \phi$) are summed over, while
the boundary edge spins, labeled $\varepsilon_i$ are invariant. In order to
obtain this from Theorem~3 of~\cite{hkice} one replaces the R-matrix $\pi(R)$ 
in the notation of that paper by a constant multiple.

\section{Bijections}
In this section we will define some combinatorial bijections that we will
need later. One of the sets is the set of states of a statistical-mechanical
system, as in the last section, and we start by defining that.

We will make use of two special $n$-tuples of integers, namely
\[\rho = (n, \cdots,3, 2, 1),\qquad
\delta = (n-1,n-2,\cdots, 2,1,0).\]
Let $\lambda = (\lambda_1, \cdots, \lambda_n)$ be a partition, so $\lambda_1
\geqslant \ldots \geqslant \lambda_n \geqslant 0$. Let us consider a lattice
with $n$ rows and $n+\lambda_1$ columns. We will index the rows from $1$ to
$n$. We will index the columns from $1$ to $n + \lambda_1$, {\tmem{in reverse
order}}. We put the vertex $v_{\Gamma} (i, j, t)$ at the vertex in the $i$
row and $j$ column.

We impose the following boundary edge spins. On the left and bottom
boundaries, every edge is labeled $+$. On the right boundary, every edge is
labeled $-$. On the top, we label the edges indexed by elements of form
$\lambda_j + n - j + 1$ for $j = 1, 2, \cdots, n$ with a $-$ spin. 
These are the entries in $\lambda+\rho$. The remaining columns we label with
a~$+$ spin.

For example, suppose that $n = 3$ and that $\lambda = (5,4,1)$, so
$\lambda+\rho=(8,6,2)$. Since
$\lambda_j + n - j + 1$ has the values $8, 6$ and $2$, we put $-$ in
these columns. We label the vertex $v_{\Gamma} (i, j, t)$ in the $i$ row and
$j$ column by $i j$ as in the following diagram:
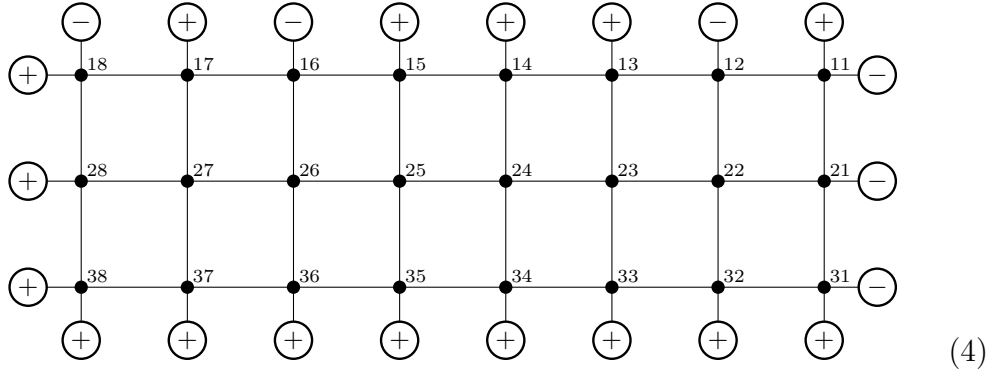
\begin{equation}
\label{iceexample}
\includegraphics{lattice.mps}
\end{equation}
Let this system be called $\mathfrak{S}^{\Gamma}_{\lambda, t}$.

We recall that a {\tmem{Gelfand-Tsetlin pattern}} is an array
\begin{equation}
  \label{lamrhopat} \mathfrak{T} = \left\{ \begin{array}{lllllll}
    p_{11} &  & p_{12} &  & \cdots &  & p_{1 n}\\
    & p_{22} &  &  &  & p_{2 n} & \\
    &  & \ddots &  & \udots &  & \\
    &  &  & p_{nn} &  &  & 
  \end{array} \right\}
\end{equation}
in which the rows are interleaving partitions. The pattern is {\tmem{strict}}
if each row is strongly dominant, meaning that $p_{i i} > p_{i, i + 1} >
\cdots > p_{i n}$.

A {\tmem{staircase}} is a semistandard Young tableau of shape
$(\lambda_1+n,\lambda_1+n-1,\ldots,\lambda_1)'$ filled with numbers from
$\{1,2,\ldots,\lambda_1+n\}$ with the additional condition that the diagonals are
weakly decreasing in the south-east direction when written in the French
notation.

An example of a staircase for $\lambda=(5,4,1)$ is

\begin{equation}
\label{stairexample}
\includegraphics{staircase.mps}
\end{equation}


%
%

\begin{proposition}\label{gtp}
Let $\lambda$ be a partition of length $\leqslant n$. There are natural bijections
between the following three sets of combinatorial objects
\begin{enumerate}
\item States of the six-vertex model $\mathfrak{S}^{\Gamma}_{\lambda, t}$,
\item Strict Gelfand-Tsetlin patterns with top row $\lambda+\rho$, and
\item Staircases whose rightmost column consists of the integers between $1$
and $\lambda_1+n$ that are not in $\lambda+\rho$.
\end{enumerate}
\end{proposition}

\begin{proof}
Suppose we start with a state of the six-vertex model $\mathfrak{S}^{\Gamma}_{\lambda, t}$. Let us record the locations
of all minus spins that live on vertical edges.
Between rows $k$ and $k+1$, there are exactly $n-k$ such minus spins for each $k$. Placing
the column numbers of the locations of these minus spins into a triangular array gives a strict Gelfand-Tsetlin pattern with top row $\lambda+\rho$.

Given a strict Gelfand-Tsetlin pattern $\mathfrak{T}=(t_{ij})$ with top row $\lambda+\rho$, we construct a staircase whose rightmost column is missing $\lambda+\rho$ in the following manner: We fill column $j+1$ with integers $u_1,\ldots,u_{n-j+\lambda_1}$ such that
\[
\{1,2,\ldots,n+\lambda_1 \} = \{u_1,\ldots ,u_{n-j+\lambda_1} \} \sqcup \{ t_{n+1-j,1},\ldots t_{n+1-j,j} \}.
\]
It is easily checked that these maps give the desired bijections.
\end{proof}

We give an example. As before let $n = 3$ and $\lambda = (5,4,1)$. Here is
an admissible state:
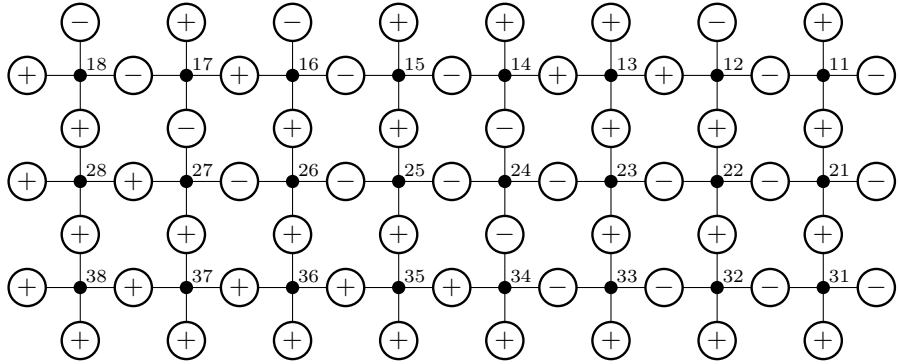
\[ \includegraphics{lattice1.mps} \]
Then the entries in the $i$-th row of $\mathfrak{T}$ are to be the
columns $j$ in which a $-$ appears above the $(i, j)$ vertex. Therefore
\[ \mathfrak{T} = \left\{ \begin{array}{lllll}
     8&&6&&2\\
     &7&&4\\
     &&4&& 
   \end{array} \right\} . \]
Taking the complements of the rows (including a fourth empty row)
gives the following sets of numbers:
\[
\begin{array}{cccccccc}
7&5&4&3&1\\8&6&5&3&2&1\\8&7&6&5&3&2&1\\8&7&6&5&4&3&2&1\end{array}\]
so the corresponding staircase is (\ref{stairexample}).

\section{A Tokuyama-like formula for Factorial Schur Functions}
\label{toksec}
A good primary reference for factorial Schur functions
is Macdonald~\cite{MacdonaldVariations}. They are also in
Macdonald~\cite{MacdonaldBook}, Ex.\ 20 in Section I.3 on p.54.

Let $\alpha_1, \alpha_2, \alpha_3, \cdots$ be a sequence of complex numbers or
formal variables. If $z \in \mathbb{C}$ let
\[ (z| \alpha)^r = (z + \alpha_1) \cdots (z + \alpha_r) . \]
Macdonald \cite{MacdonaldVariations} gives two formulas for factorial Schur
functions that we will also prove to be equivalent by our methods. Let 
$\mu = (\mu_1, \cdots, \mu_n)$ where the $\mu_i$ are nonnegative integers. Let $z_1,
\cdots, z_n$ be given. Define
\[ A_{\mu} (z| \alpha) = \det ((z_i | \alpha)^{\mu_j})_{i, j} \]
where $1 \leqslant i, j \leqslant n$ in the determinant. We will also
use the notation
\[\tmmathbf{z}^\mu=\prod_i z_i^{\mu_i}.\]

Let $\delta = (n - 1, n - 2, \cdots, 0)$ and let $\lambda = (\lambda_1, \cdots,
\lambda_n)$ be a partition of length at most $n$. Define
\begin{equation}
\label{fsfdef}
s_{\lambda} (z| \alpha) = \frac{A_{\lambda + \delta} (z| \alpha)}{A_{\delta}
   (z| \alpha)} .
\end{equation}
The denominator here is actually independent of $\alpha$ and is
given by the Weyl denominator formula:
\begin{equation}\label{arho}
A_\delta(z|\alpha)=\prod_{i<j}(z_i-z_j)
\end{equation}
Indeed, it is an alternating polynomial in the $z_i$ of the same degree
as the right-hand side, and so the ratio is a polynomial in $\alpha$ 
that is independent of $z_i$. To see that the ratio is independent
of $\alpha$, one may compare the coefficients of 
$z^\delta$ on both sides of (\ref{arho}). Since both the 
numerator and the denominator is an alternating function of the $z_i$, the
ratio $s_{\lambda} (z| \alpha)$ is a symmetric polynomial in $z_1, \cdots,
z_n$.

Now let us consider two types of Boltzmann weights. Let $z_1, \cdots, z_n$ be
given, and let $\alpha_1, \alpha_2, \alpha_3, \cdots$ another sequence
of complex numbers, and $t$ another parameter. If $1 \leqslant i, k \leqslant
n$ and if $j \geqslant 0$ are integers, we will use the following weights.
\[ \text{$\begin{array}{|c|c|c|c|c|c|c|}
     \hline
     v_{\Gamma} (i, j, t) & \includegraphics{gamma1a.mps} &
     \includegraphics{gamma6a.mps} & \includegraphics{gamma4a.mps} &
     \includegraphics{gamma5a.mps} & \includegraphics{gamma2a.mps} &
     \includegraphics{gamma3a.mps}\\
     \hline
     & 1 & z_i - t \alpha_j & t & z_i + \alpha_j & z_i (t + 1) &
     1\\
     \hline
     v_{\Gamma \Gamma} (i, k, t) & \includegraphics{atom1c.mps} &
     \includegraphics{atom6c.mps} & \includegraphics{atom4c.mps} &
     \includegraphics{atom3c.mps} & \includegraphics{atom2c.mps} &
     \includegraphics{atom5c.mps}\\
     \hline
     & tz_i + z_k & tz_k + z_i & t (z_k - z_i) & z_i - z_k & (t + 1) z_i & (t
     + 1) z_k\\
     \hline
   \end{array}$} \]
\begin{lemma}
  We may take $u = v_{\Gamma \Gamma} (i, k, t)$, $v = v_{\Gamma} (i, j, t)$
  and $w = v_{\Gamma} (k, j, t)$ in Proposition~\ref{ybe}.
\end{lemma}

\begin{proof}
  The relation $a_1 (u) = a_1 (v) a_2 (w) + b_2 (v) b_1 (w)$ becomes
  \[ t z_i + z_k = 1 \cdot (z_k - t \alpha_j) + (z_i + \alpha_j)
     t, \]
  and all the other relations are checked the same way.
\end{proof}

\begin{proposition}
  \label{symprop}The function
  \begin{equation}
    \label{multipliedz} \left[ \prod_{i > j} (t z_j + z_i) \right] Z
    (\mathfrak{S}_{\lambda,t}^\Gamma),
  \end{equation}
  is symmetric under permutations of the $z_i$.
\end{proposition}

\begin{proof}
  Let
  $\mathfrak{S}=\mathfrak{S}_{\lambda, t}^{\Gamma}$. It is sufficient to show
  that (\ref{multipliedz}) is invariant under the interchange of $z_i$ and
  $z_{i + 1}$. The factors in front are permuted by this interchange with one
  exception, which is that $t z_i + z_{i + 1}$ is turned into $t z_{i + 1} +
  z_i$. Therefore we want to show that $(t z_i + z_{i + 1}) Z (\mathfrak{S}) =
  (t z_{i + 1} + z_i) Z (\mathfrak{S}')$ where $\mathfrak{S}'$ is the system
  obtained from $\mathfrak{S}$ by interchanging the weights in the $i, i + 1$
  rows. Now consider the modified system obtained from $\mathfrak{S}$ by
  attaching $v_{\Gamma \Gamma} (i, i + 1, t)$ to the left of the $i, i + 1$ rows.
  For example, if $i = 1$, this results in the following system:
  
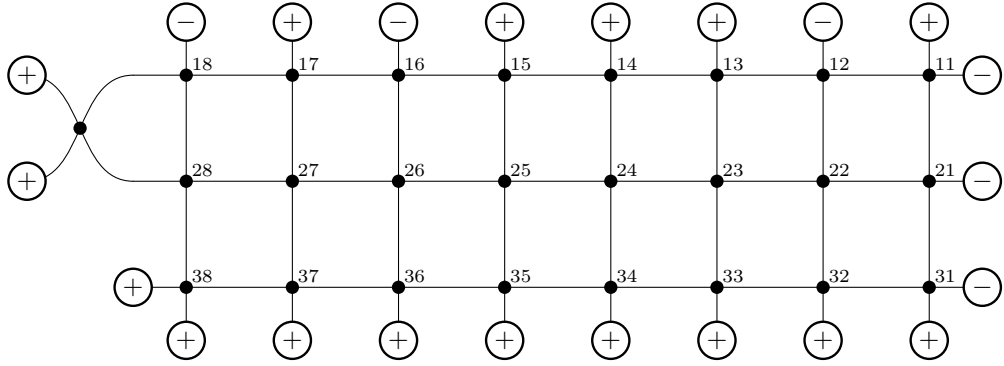
\[ \includegraphics{lattice3.mps} \]
  Referring to (\ref{boltzmannweights}), there is only one admissible
  configuration for the two interior edges adjoining the new vertex,
  namely both must be $+$, and so the Boltzmann weight of this vertex
  will be $t z_i+z_{i+1}$. Thus the partition function of this new system
  equals $(t z_i + z_{i + 1}) Z (\mathfrak{S})$. Applying the Yang-Baxter
  equation repeatedly, this equals the partition function of
  the system obtained from $\mathfrak{S}'$ by adding $v_{\Gamma \Gamma} (i, i
  + 1, t)$ to the right of the $i, i + 1$ rows, that is, $\text{$(t z_{i + 1}
  + z_i) Z (\mathfrak{S}')$}$, as required. This proves that
  (\ref{multipliedz}) is symmetric.
\end{proof}

\begin{theorem}
  \label{icerepn}We have
  \[ Z (\mathfrak{S}^{\Gamma}_{\lambda, t}) = \left[ \prod_{i < j} (t z_j +
     z_i) \right] s_{\lambda} (z| \alpha) . \]
\end{theorem}

\begin{proof}
  We show that the ratio
  \begin{equation}
    \label{shiftratio} \frac{Z (\mathfrak{S}^{\Gamma}_{\lambda, t})}{\prod_{i
    < j} (t z_j + z_i)}
  \end{equation}
  is a polynomial in the $z_i$, and that it is independent of $t$. Observe
  that (\ref{multipliedz}) is an element of the polynomial ring
  $\mathbb{C}[z_1, \cdots, z_n, t]$, which is a unique factorization domain.
  It is clearly divisible by $t z_j + z_i$ when $i > j$, and since it is
  symmetric, it is therefore divisible by all $t z_j + z_i$ with $i \neq j$.
  These are coprime, and therefore it is divisible by their product, in other
  words (\ref{shiftratio}) is a polynomial. We note that the numerator and the
  denominator have the same degree in $t$, namely $\frac{1}{2} n (n - 1)$. For
  the denominator this is clear and for the numerator, we note that each term
  is a monomial whose degree is the number of vertices with a $-$ spin on the
  vertical edge below. This is the number of vertical edges labeled $-$
  excluding those at the top, that is, the number of entries in the
  Gelfand-Tsetlin pattern in Lemma~\ref{gtp} {\tmem{excluding}} the first row.
  Therefore each term in the sum $Z (\mathfrak{S}_{\lambda, t}^{\Gamma})$ is a
  monomial of degree $ \frac{1}{2} n (n - 1)$ and so the sum has at most this
  degree; therefore the ratio (\ref{shiftratio}) is a polynomial of degree $0$
  in $t$, that is, independent of $t$.
  
  To evaluate it, we may choose $t$ at will. We take $t = - 1$.
  
  Let us show that if the pattern occurs with nonzero Boltzmann weight, then
  every row of the pattern (except the top row) is obtained from the row above
  it by discarding one element. Let $\mu$ and $\nu$ be two partitions that
  occur as consecutive rows in this Gelfand-Tsetlin pattern:
  \[ \left\{ \begin{array}{lllllllll}
       \mu_1 &  & \mu_2 &  & \cdots &  & \cdots &  & \mu_{n - i + 1}\\
       & \nu_1 &  & \nu_2 &  & \cdots &  & \nu_{n - i} & 
     \end{array} \right\}, \]
  where $\mu_k$ are the column numbers of the vertices $(i, \mu_k)$ that have a
  $-$ spin on the edge above the vertex, and $\nu_k$ are the column numbers of the
  vertices $(i, \nu_k)$ that have a $-$ spin on the edge below it. Because $t = - 1$,
  the pattern
  
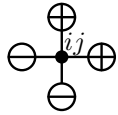
\[ \includegraphics{gamma2a.mps} \]
  does not occur, or else the Boltzmann weight is zero, and the term may be
  discarded. Therefore every $-$ spin below the vertex in the $i$-th row must be
  matched with a $-$ above the vertex. It follows that every $\nu_i$ equals
  either $\mu_i$ or $\mu_{i + 1}$. Thus the partition $\nu$ is obtained from
  $\mu$ by discarding one element.
  
  Let $\mu_k$ be the element of $\mu$ that is not in $\nu$. It is easy to see
  that in the horizontal edges in the $i$-th row, we have a $-$ spin to the right of
  the $\mu_k$-th column and $+$ to the left. Since $t = - 1$, the Boltzmann
  weights of the two patterns:
  \[ \begin{array}{l}
     \end{array} 
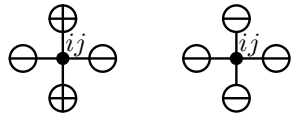
\includegraphics{gamma5a.mps} \hspace{2em}
     \includegraphics{gamma6a.mps} \]
  both have the same Boltzmann weight $z_i + \alpha_{j}$. We have one such
  contribution from every column to the right of $\mu_k$ and these contribute
  \[ \prod_{j = 1}^{\mu_k - 1} (z_i + \alpha_{j}) = (z_i | \alpha)^{\mu_k-1}
     . \]
  
  For each column $j = \mu_l$ with $l < k$ we have a pattern
  
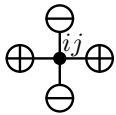
\[ \includegraphics{gamma4a.mps} \]
  and these contribute $- 1$. Therefore the product of the Boltzmann weights
  for this row is
  \[ (- 1)^{k - 1} (z_i | \alpha)^{\mu_k-1} . \]
  
  Since between the $i$-th row $\mu$ and the $(i+1)$-st row $\nu$ of the
  Gelfand-Tsetlin pattern one element is discarded, there is some permutation
  $\sigma$ of $\{1, 2, 3, \cdots, n\}$ such that $\nu$ is obtained
  by dropping the $\sigma (i)$-th element of $\mu$. In other words, 
\[\mu_k-1 = (\lambda + \rho)_{\sigma (i)}-1 = (\lambda + \delta)_{\sigma (i)}\]
  and we conclude that
  \[ Z (\mathfrak{S}_{\lambda, - 1}^{\Gamma}) = \sum_{\sigma \in S_n} \pm
     \prod_i (z_i | \alpha)^{(\lambda + \delta)_{\sigma (i)}} . \]
  The signs may be determined as follows. First, take all $\alpha_i = 0$ and
  $t=-1$. The term
  corresponding to $\sigma$ is then $\pm \prod_i z_i^{(\lambda + \delta)_{\sigma
  (i)}}$. The ratio (\ref{shiftratio}) is symmetric, and with $t = - 1$ the
  denominator is antisymmetric. This shows
  \[ Z (\mathfrak{S}_{\lambda, - 1}^{\Gamma}) = \pm \sum_{\sigma \in S_n} (-
     1)^{l (\sigma)} \prod_i (z_i | \alpha)^{(\lambda + \delta)_{\sigma (i)}} = \pm
     A_{\lambda + \delta} (z|\alpha) . \]
  We still need to determine the leading $\pm$.
  Using (\ref{arho}) we see that (\ref{shiftratio}) equals $\pm
  s_{\lambda} (z| \alpha)$. To evaluate the sign, we may take $t = 0$ and all
  $z_i = 1$. Then the partition function is a sum of positive terms, and
  $s_{\lambda} (1, \cdots, 1)$ is positive, proving that (\ref{shiftratio})
  equals $s_{\lambda} (z| \alpha)$.
\end{proof}

\section{The Combinatorial Definition of Factorial Schur Functions\label{thmtwo}}

Macdonald also gives a combinatorial formula for the factorial Schur function
as a sum over semi-standard Young tableaux of shape $\lambda$ in 
$\{1, 2, \cdots, n\}$. Let $T$ be such a tableau. This generalizes the
well-known combinatorial formula for Schur functions. If $i, j$ are given such that $j
\leqslant \lambda_i$, let $T (i, j)$ be the entry in the $i$-th row and $j$-th
column of $T$. Let
\begin{equation}
  T^{\ast} (i, j) = T (i, j) + j - i \label{tstardef},
\end{equation}
and define
\begin{equation}
\label{zaltdef}
(z| \alpha)^T = \prod_{(i, j)} (z_{T (i, j)} + \alpha_{T^{\ast} (i, j)}) .
\end{equation}
For example, if $\lambda = (4, 2, 0)$ and $n = 3$, then we might have
\begin{equation}
  \label{tableauxexample} T = \begin{array}{l}
    \begin{array}{|l|l|l|l|}
      \hline
      1 & 1 & 1 & 3\\
      \hline
    \end{array}\\
    \begin{array}{|l|l|}
      \hline
      2 & 2\\
      \hline
    \end{array}
  \end{array}, \hspace{2em} T^{\ast} = \begin{array}{l}
    \begin{array}{|l|l|l|l|}
      \hline
      1 & 2 & 3 & 6\\
      \hline
    \end{array}\\
    \begin{array}{|l|l|}
      \hline
      1 & 2\\
      \hline
    \end{array}
  \end{array}
\end{equation}
and
\[ (z| \alpha)^T = (z_1 + \alpha_1) (z_1 + \alpha_2) (z_1 + \alpha_3) (z_3 +
   \alpha_6) (z_2 + \alpha_1) (z_2 + \alpha_2) . \]
\begin{theorem}
  \label{theoremtwo}Let $\lambda$ be a partition. Then
  \begin{equation}
    \label{tableaudef}
    s_{\lambda} (z| \alpha) = \sum_T (z| \alpha)^T,
  \end{equation}
  where the sum is over semistandard Young tableaux with shape $\lambda$ in
  $1, 2, 3, \cdots, n$.
\end{theorem}

This formula, expressing the factorial Schur function as a sum over
semistandard Young tableaux, is equivalent (in a special case) to a formula
of Biedenharn and Louck~\cite{BiedenharnLouck}, who made it the definition of the factorial Schur
function. For them, the sum was over the set of Gelfand-Tsetlin patterns with
prescribed top row, but this is in bijection with tableaux. In this
generality, the formula is due to Macdonald~\cite{MacdonaldVariations}.

When $t = 0$, the Boltzmann weight for \begin{tabular}{l}
  $\includegraphics{gamma4a.mps}$
\end{tabular}is zero, so we are limited to states omitting this configuration.
If $\mathfrak{T}=\mathfrak{T}_{\lambda + \rho}$ is the Gelfand-Tsetlin pattern
corresponding to this state, and if the entries of $\mathfrak{T}$ in are
denoted $p_{i k}$ as in (\ref{lamrhopat}), it is easy to see that the equality
$p_{i - 1, k - 1} = p_{i, k}$ would cause this configuration to appear at the
$i, j$ position, where $j = p_{i, k}$. Therefore $p_{i - 1, k - 1} > p_{i,
k}$. This inequality implies that we may obtain another Gelfand-Tsetlin
pattern $\mathfrak{T}_{\lambda}$ with top row $\lambda$ by subtracting
$\rho_{n - i + 1} = (n - i+1, n - i , \cdots, 1)$ from the $i$-th row of
$\mathfrak{T}_{\lambda + \rho}$.

Consider the example of $\lambda = (4, 2, 0)$ with $n = 3$. Then
\[ \text{if{\hspace{1em}}$\mathfrak{T}_{\lambda + \rho} = \left\{
   \begin{array}{ccccc}
     7 &  & 4 &  & 1\\
     & 5 &  & 3 & \\
     &  & 4 &  & 
   \end{array}
   \right\}${\hspace{1em}}then{\hspace{1em}}$\mathfrak{T}_{\lambda} = \left\{
   \begin{array}{ccccc}
     4 &  & 2 &  & 0\\
     & 3 &  & 2 & \\
     &  & 3 &  & 
   \end{array} \right\}$} . \]
We associate with $\mathfrak{T}_{\lambda}$ a tableau $T (
\mathfrak{T}_{\lambda})$ of shape $\lambda$. In this tableau, removing all
boxes labeled $\begin{array}{|l|}
  \hline
  n\\
  \hline
\end{array}$ from the diagram produces a tableau whose shape is the second row
of $\mathfrak{T}_{\lambda}$. Then removing boxes labeled $\begin{array}{|l|}
  \hline
  n - 1\\
  \hline
\end{array}$ produces a tableau whose shape is the third row of
$\mathfrak{T}_{\lambda}$, and so forth. Thus in the example, $T
(\mathfrak{T}_{\lambda})$ is the tableau $T$ in (\ref{tableauxexample}).
Let $w_0$ denote the long element of the Weyl group $S_n$, which is the
permutation $i\mapsto n+1-i$.

\begin{proposition}
  Let $t = 0$, and let $\mathfrak{s}=\mathfrak{s}( \mathfrak{T}_{\lambda +
  \rho})$ be the state corresponding to a special Gelfand-Tsetlin pattern
  $\mathfrak{T}_{\lambda + \rho}$. 
  With $\mathfrak{T}_{\lambda}$ as above and
  $T = T ( \mathfrak{T}_{\lambda})$ we have
  \begin{equation}
    \label{wostate} w_0 \left( \prod_{v \in \mathfrak{s}} \beta_{\mathfrak{s}}
    (v) \right) = \tmmathbf{z}^{w_0 (\delta)} (z|a)^T .
  \end{equation}
\end{proposition}

Before describing the proof, let us give an example. With
$\mathfrak{T}_{\lambda + \rho}$ as above, the state $\mathfrak{s}(
\mathfrak{T}_{\lambda + \rho})$ is
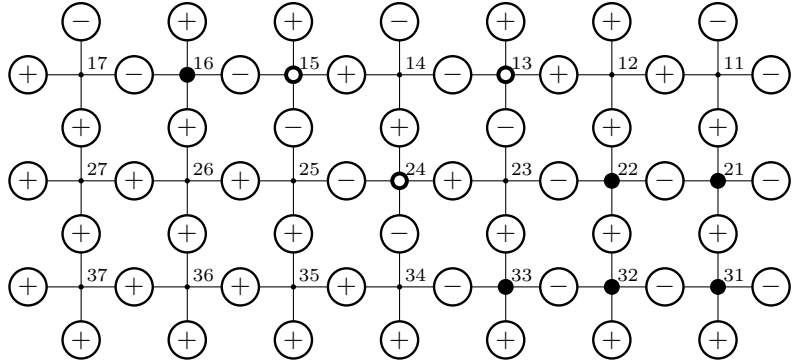
\[ \includegraphics{lattice2.mps} \]
The locations labeled $\circ$ produce powers of $z_i$, and the locations
labeled $\bullet$ produce shifts of the form $z_i + \alpha_{j}$. The
weight of this state is
\[ \prod_{v \in \mathfrak{s}} \beta_{\mathfrak{s}} (v) = z_1^2 z_2 (z_1 +
   \alpha_6) (z_2 + \alpha_2) (z_2 + \alpha_1) (z_3 + \alpha_3) (z_3 +
   \alpha_2) (z_3 + \alpha_1) . \]
Applying $w_0$ interchanges $z_i \longleftrightarrow z_{4 - i}$. The factor
$z_1^2 z_2$ becomes $z_3^2 z_2 = \tmmathbf{z}^{w_0 (\delta)}$, and the terms
that remain agree with
\[ \prod_{(i, j)} (z_{T (i, j)} + \alpha_{T^{\ast} (i, j)}) . \]
\begin{proof}
  We are using the following Boltzmann weights.
  \[ \begin{array}{|c|c|c|c|c|c|c|}
       \hline
       \begin{array}{c}
         \tmop{Gamma}\\
         \text{Ice}
       \end{array} & \includegraphics{gamma1a.mps} &
       \includegraphics{gamma6a.mps} & \includegraphics{gamma4a.mps} &
       \includegraphics{gamma5a.mps} & \includegraphics{gamma2a.mps} &
       \includegraphics{gamma3a.mps}\\
       \hline
       \text{\begin{tabular}{c}
         Boltzmann\\
         weight 
       \end{tabular}} & 1 & z_i & 0 & z_i + \alpha_{j} & z_i & 1\\
       \hline
     \end{array} \]
  We have contributions of $z_i$ from vertices that have $-$ on the vertical
  edge below, and there is one of these for each entry in the Gelfand-Tsetlin
  pattern. These contribute a factor of $\tmmathbf{z}^{\delta}$. Applying $w_0$
  to $\tmmathbf{z}$ leads to $\tmmathbf{z}^{w_0 (\delta)}$.
  
  Considering the contribution from the $(i, j)$ vertex, there will be a
  factor of $z_i + \alpha_{j}$ when the vertex has the configuration
  \begin{tabular}{l}
    \includegraphics{gamma5a.mps}
  \end{tabular}. In the above example, the locations are labeled by $\bullet$.
  
  Let $\mathfrak{T}_{\lambda + \rho}$ be the Gelfand-Tsetlin pattern
  (\ref{lamrhopat}). Let $\mathfrak{T}_{\lambda}$ and $T = T
  (\mathfrak{T}_{\lambda})$ be as described above, and let $T^{\ast}$ be as in
  (\ref{tstardef}). Let the entries in $\mathfrak{T}_{\lambda+\rho}$ and
  $\mathfrak{T}_\lambda$ be denoted $p_{i,j}$ and $q_{i,j}$, with the
  indexing as in (\ref{lamrhopat}). Thus $q_{i,j}=p_{i,j}-n+j-1$. We will
  make the convention that $q_{i,n+1}=p_{i,n+1}=0$. 
  The condition for $\bullet$ in the $i, j$ position of the
  state is that for some $k$ with $i\leqslant k\leqslant n$ we have
  $p_{i+1,k+1} < j < p_{i,k}$.
  Translating this in terms of the $q_{i,j}=p_{i,j}-n+j-1$ the condition
  becomes $q_{i+1,k+1}+n-k+1 \leqslant j < q_{i,k}+n-k$. The effect of
  $w_0$ is to interchange $z_i \leftrightarrow z_{n - i + 1}$, and therefore
\begin{equation}
\label{wobw}
w_0 \left( \prod_{v \in \mathfrak{s}} \beta_{\mathfrak{s}} (v) \right)=
\tmmathbf{z}^{w_0(\delta)}\prod_{i=1}^n\quad\prod_{k=i}^n\quad\prod_{j=q_{i+1,k+1}+n-k+1}^{q_{i,k}+n-k}(z_{n+1-i}+\alpha_j).
\end{equation}
  On the other hand, in the tableau $T$, the location of the entries
  equal to $\begin{array}{|l|}\hline n+1-i\\\hline\end{array}$ in the
  $(k+1-i)$-th row is between columns $q_{i+1,k+1}+1$ through $q_{i,k}$,
  and if $j$ is one of these columns then
\[T(k+1-i,j)=n+1-i,\qquad T^*(k+1-i,j)=n+j-k.\]
 Therefore in the notation (\ref{zaltdef}) we have
\[(z| \alpha)^T=\prod_{i=1}^n\quad\prod_{k=i}^n\prod_{j=q_{i+1,k+1}+1}^{q_{i,k}}(z_{n+1-i}+\alpha_{n+j-k}).\]
This equals (\ref{wobw}) and the proof is complete.
\end{proof}

We now give the proof of Theorem~\ref{theoremtwo}. Summing over states, the
last Proposition implies that
\[ Z_{\lambda} (w_0 ( \tmmathbf{z}), \alpha, 0) = \tmmathbf{z}^{w_0 (\delta)}
   \sum_T (z|a)^T . \]
Since $s_{\lambda}^{\Gamma} ( \tmmathbf{z}, \alpha, 0) = s_{\lambda}^{\Gamma}
(w_0 ( \tmmathbf{z}), \alpha, 0)$ we have
\[ s_{\lambda}^{\Gamma} (z, \alpha, 0) = \frac{Z_{\lambda} (w_0 (
   \tmmathbf{z}), \alpha, 0)}{\prod_{i > j} w_0 ( \tmmathbf{z})_j} =
   \frac{\tmmathbf{z}^{w_0 (\delta)} \sum_T (z|a)^T}{\tmmathbf{z}^{w_0 (\delta)}}
   = \sum_T (z|a)^T, \]
and the statement follows.

\section{The limit as t tends to infinity}\label{thmthree}

Let $\mu$ be a partition, and let $\alpha_{\mu}$ denote the sequence
\[ \alpha_{\mu} = (\alpha_{\mu_1 + n}, \alpha_{\mu_2 + n - 1}, \cdots,
   \alpha_{\mu_n + 1}) . \]
If $\lambda$ is a partition then $\lambda'$ will denote the conjugate
partition whose Young diagram is the transpose of that of $\lambda$.

\begin{theorem}
  \label{theoremthree}{\dueto{Vanishing Theorem}}We have
  \[ s_{\lambda} (- \alpha_{\mu} | \alpha) = \left\{ \begin{array}{ll}
       0 & \text{if $\lambda \not\subset \mu$,}\\
       \prod_{(i, j) \in \lambda} (\alpha_{n - i + \lambda_i + 1} - \alpha_{n
       - \lambda'_j + j}) & \text{if $ \lambda=\mu$.}
     \end{array} \right. \]
\end{theorem}

See Okounkov~{\cite{Okounkov}} Section~2.4 and Molev and
Sagan~{\cite{MolevSagan}}. In view of the relationship between Schubert
polynomials for Grassmannian permutations and factorial Schur functions,
this is equivalent to an older vanishing statement for Schubert polynomials.
Vanishing properties for Schubert polynomials are implied by
Theorem~9.6.1 and Proposition~9.6.2 of Lascoux~\cite{LascouxBook},
which are related to the results of Lascoux and
Sch\"utzenberger~\cite{LascouxSchutz1, LascouxSchutz2}.
We will prove it in this section by our methods.

We examine the behavior of our six-vertex model as we send the parameter $t$ to infinity. 
The first result of this section may be construed as a rederivation of a
theorem of Lascoux \cite[Theorem 1]{LascouxSchubert}. However the approach we
take is to interpret it as giving us a proof of the equivalence of factorial
Schur functions and double Schubert polynomials for Grassmannian
permutations. We also obtain the vanishing theorem for factorial Schur
functions (Theorem~\ref{theoremthree}).

We start with two simple lemmas. To state these, we describe the six admissible arrangements of edge spins around a vertex as types
$a_1$, $a_2$, $b_1$, $b_2$, $c_1$ and $c_2$ respectively, when reading from left to right in the diagram (\ref{boltzmannweights}).

\begin{lemma}\label{bce}
In each state, the total number of sites of type $a_2$, $b_1$ and $c_1$ is equal to $n(n-1)/2$.
\end{lemma}

\begin{proof}
This number is equal to the number of minus spins located in the interior of a vertical string. 
\end{proof}

Let $\mu$ be the partition $(\lambda+\delta)'$.

\begin{lemma}\label{bde}
In each state, the number of occurrences of $a_2$, $b_2$ and $c_1$ patterns in the $i$-th column is equal to $\mu_i$.
\end{lemma}
\begin{proof}
This number is equal to the number of minus spins located on a horizontal
string between the strings labeled $i$ and $i+1$.  This count is known since
for any rectangle, knowing the boundary conditions on the top, bottom and
rightmost sides determines the number of such spins along the leftmost edge.
\end{proof}

As a consequence of Lemma \ref{bce}, the result of taking the limit as
$t\rightarrow\infty$ can be interpreted as taking the leading degree term in
$t$ for each of the Boltzmann weights $v_\Gamma (i,j,t)$ as in Section
\ref{toksec}. Thus with the set of Boltzmann weights $(1, -\alpha_j,
1,z_i+\alpha_j, z_i, 1)$ and corresponding partition function
$Z(\mathfrak{S}^\Gamma_{\lambda,\infty})$, we have
$$
Z(\mathfrak{S}^\Gamma_{\lambda,\infty})=z^\delta s_\lambda(z|\alpha).
$$

By Lemma \ref{bde}, if we consider our ice model with the series of Boltzmann
weights in the following diagram, then the corresponding partition function is
given by dividing $Z(\mathfrak{S}^\Gamma_{\lambda,\infty})$ by
$(-\alpha)^{\mu}$.

 \[ \begin{array}{|c|c|c|c|c|c|c|}
       \hline
       \begin{array}{c}
         \tmop{Gamma}\\
         \text{Ice}
       \end{array} & \includegraphics{gamma1a.mps} &
       \includegraphics{gamma6a.mps} & \includegraphics{gamma4a.mps} &
       \includegraphics{gamma5a.mps} & \includegraphics{gamma2a.mps} &
       \includegraphics{gamma3a.mps}\\
       \hline
       \text{\begin{tabular}{c}
         Boltzmann\\
         weight 
       \end{tabular}} & 1 & 1 & 1 & -z_i / \alpha_{j}-1 & -z_i/\alpha_j & 1\\
       \hline
     \end{array} \]
Let us denote the partition function for this set of Boltzmann weights by
$Z(\mathfrak{S}^{\Gamma'}_{\lambda,\infty}(z|\alpha))$. Thus we obtain the
result of \cite[Theorem 1.1]{McNamara},
\begin{equation}\label{infinity}
Z(\mathfrak{S}^{\Gamma'}_{\lambda,\infty} (z|\alpha))=
\frac{\tmmathbf{z}^\delta}{(-\tmmathbf{\alpha})^{(\lambda+\delta)'}}s_\lambda(z|\alpha).
\end{equation}

We now pause to introduce the notions of double Schubert polynomials and Grassmannian permutations
 so that we can make the connection to \cite[Theorem 1]{LascouxSchubert} precise.


A permutation is {\tmem{Grassmannian}} if it has a unique (right) descent.

Let $n$ and $m$ be given and let $\lambda = (\lambda_1, \cdots, \lambda_n)$ be
a partition such that $\lambda_1 \leqslant m$. Then there is associated with
$\lambda$ a Grassmannian permutation $w_{\lambda} \in S_{n + m}$. This is the permutation such that
\[ w_{\lambda} (i) = \left\{ \begin{array}{ll}
     \lambda_{n + 1 - i} + i & \text{if $i \leqslant n$,}\\
     i - \lambda_{i - n}' & \text{if $i > n$,}
   \end{array} \right. \]
where $\lambda'$ is the conjugate partition. This has $w_{\lambda} (n + 1) <
w_{\lambda} (n)$ and no other descent.

Let $x_1, \cdots, x_{n + m}$ and $y_1, \cdots, y_{n + m}$ be parameters. We
define the divided difference operators as follows. If $1 \leqslant i < n + m$
and $f$ is a function of the $x_i$ let
\[ \partial_i f (x_1, \cdots, x_{n + m}) = \frac{f - s_i f}{x_i - x_{i + 1}}
\]
where $s_i f$ is the function obtained by interchanging $x_i$ and $x_{i + 1}$.
Then if $w \in S_{n + m}$, let $w = s_{i_1} \cdots s_{i_k}$ be a reduced
expression of $w$ as a product of simple reflections. Then let $\partial_w =
\partial_{i_1} \cdots \partial_{i_k}$. This is well-defined since the divided
difference operators $\partial_i$ satisfy the braid relations. Let $w_0$ be
the long element of $S_{n + m}$. Then the \textit{double Schubert
  polynomials},
which were defined by Lascoux and Sch\"utzenberger~\cite{LascouxSchutz} are
given by
\[ \mathfrak{S}_w (x, y) = \partial_{w^{- 1} w_0} \left( \prod_{i + j
   \leqslant n + m} (x_i - y_j) \right) . \]

The theory of factorial Schur functions is a special case of the
theory of double Schubert polynomials developed by Lascoux and
Sch\"utzenberger \cite{LascouxSchutz, LascouxSchutz1}. Although the comparison
is well-known, there does not appear to be a truly satisfactory
reference in the literature. We give a new proof in the thematic
spirit of this paper.

\begin{theorem}
The factorial Schur functions are equal to double Schubert polynomials for
Grassmannian permutations.
More precisely,
\[\mathfrak{S}_{w_\lambda}(x,y)=s_\lambda(x|-y)\]
\end{theorem}

\medbreak
\begin{proof}
The proof is a comparison of (\ref{infinity}) with \cite[Theorem 1]{LascouxSchubert}.
To translate the left-hand side in (\ref{infinity}) into the staircase
language of Lascoux we use the bijection in Proposition \ref{gtp}. 
In Theorem~1 of \cite{LascouxSchubert} if Lascoux' $x$ is our $z$
and his $y$ is our $-\alpha$, then the left-hand side of his identity
is exactly the partition function on the left-hand side of our (\ref{infinity}).
The monomial $x^{\rho_r}y^{-\left<\tilde{u}\right>}$
on the right-hand side of his identity equals the monomial
$z^\rho(-\alpha)^{-(\lambda+\delta)'}$
on the right-hand side of (\ref{infinity}). Lascoux'
$\mathbb{X}_{\tilde{u},u^\omega}$
is the double Schubert polynomial $\mathfrak{S}_{w_\lambda}$. The
statement follows.
\end{proof}

To conclude this section, we shall use this description to give a proof of the
characteristic vanishing property of Schur functions. In lieu of
(\ref{infinity}) above, the following result is clearly equivalent to Theorem
\ref{theoremthree}.

\begin{theorem}
For two partitions $\lambda$ and $\mu$ of at most $n$ parts, we have
$$Z(\mathfrak{S}^{\Gamma'}_{\lambda,\infty}(-\alpha_\mu|\alpha))=0 \text{ unless } \lambda\subset\mu ,$$
$$Z(\mathfrak{S}^{\Gamma'}_{\lambda,\infty}(-\alpha_\lambda|\alpha))=\prod_{(i,j)\in\lambda} \left(\frac{\alpha_{n+1-i+\lambda_i}}{\alpha_{n-\lambda'_j+j}}-1\right).$$
\end{theorem}
\begin{proof}
Fix a state, and assume that this state gives a non-zero contribution to the
partition function $Z(\mathfrak{S}^{\Gamma'}_{\lambda,\infty}(-\alpha_\mu|\alpha))$. Under the bijection between states of
square ice and strict Gelfand-Tsetlin patterns, let $k_i$ be the leftmost
entry in the $i$-th row of the corresponding Gelfand-Tsetlin pattern. We shall
prove by descending induction on $i$ the inequality
$$
n+1-i+\mu_i\geq k_i.
$$
For any $j$ such that $k_{i+1}<j<k_i$, there is a factor $(x_i/\alpha_j-1)$ in the
Boltzmann weight of this state. We have the inequality
$n+1-i+\mu_i>n+1-(i+1)+\mu_{i+1}\geq k_{i+1}$ by our inductive
hypothesis. Since $(\alpha_\mu)_i=\alpha_{n+1-i+\mu_i}$, in order for this state to give
a non-zero contribution to $Z(\mathfrak{S}^{\Gamma'}_{\lambda,\infty}(-\alpha_\mu|\alpha))$, we must have that
$n+1-i+\mu_i\geq k_i$, as required.

Note that for all $i$, we have $k_i\geq n+1-i+\lambda_i$. Hence $\mu_i\geq\lambda_i$ for all $i$, showing that $\mu\supset\lambda$ as required, proving the first part of the theorem.

To compute $Z(\mathfrak{S}^{\Gamma'}_{\lambda,\infty}(-\alpha_\lambda|\alpha))$, notice that the above argument shows that there is only one state which gives a non-zero contribution to the sum. Under the bijection with Gelfand-Tsetlin patterns, this is the state with 
$p_{i,j}=p_{1,j}$ for all $i,j$. The formula for $Z(\mathfrak{S}^{\Gamma'}_{\lambda,\infty}(-\alpha_\lambda|\alpha))$ is now immediate.
\end{proof}

\section{Asymptotic Symmetry}
Macdonald~{\cite{MacdonaldVariations}} shows that the factorial Schur
functions are {\tmem{asymptotically symmetric}} in the $\alpha_i$ as the
number of parameters $z_i$ increases. To formulate this property,
let $\sigma$ be a permutation of the parameters $\alpha = (\alpha_1, \alpha_2,
\cdots)$ such that $\sigma(\alpha_j)=\alpha_j$ for all but finitely many $j$.
Then we will show that if the number of parameters $z_i$ is
sufficiently large, then $s_{\lambda} (z| \sigma \alpha) = s_{\lambda} (z|
\alpha)$. How large $n$ must be depend on both $\lambda$ and on the
permutation $\sigma$.

We will give a proof of this symmetry property of factorial Schur functions
using the Yang-Baxter equation. In the following theorem we will use the
following Boltzmann weights:
\[ \begin{array}{|c|c|c|c|c|c|c|}
     \hline
     v & \includegraphics{weighta1.mps} & \includegraphics{weighta2.mps} &
     \includegraphics{weightb1.mps} & \includegraphics{weightb2.mps} &
     \includegraphics{weightc1.mps} & \includegraphics{weightc2.mps}\\
     \hline
     & 1 & z_i - t \alpha & t & z_i + \alpha & z_i (t + 1) & 1\\
     \hline
     w & \includegraphics{weighta1.mps} & \includegraphics{weighta2.mps} &
     \includegraphics{weightb1.mps} & \includegraphics{weightb2.mps} &
     \includegraphics{weightc1.mps} & \includegraphics{weightc2.mps}\\
     \hline
     & 1 & z_i - t \beta & t & z_i + \beta & z_i (t + 1) & 1\\
     \hline
     u & \includegraphics{rota2.mps} & \includegraphics{rota1.mps} &
     \includegraphics{rotb1.mps} & \includegraphics{rotb2.mps} &
     \includegraphics{rotc1b.mps} & \includegraphics{rotc2b.mps}\\
     \hline
     & 1 & 1 & \alpha - \beta & 0 & 1 & 1\\
     \hline
   \end{array} \]
\begin{theorem}
  \label{vertybe}With the above Boltzmann weights, and with $\epsilon_1,
  \cdots, \epsilon_6$ fixed spins $\pm$, the following two systems have the
  same partition function.
  
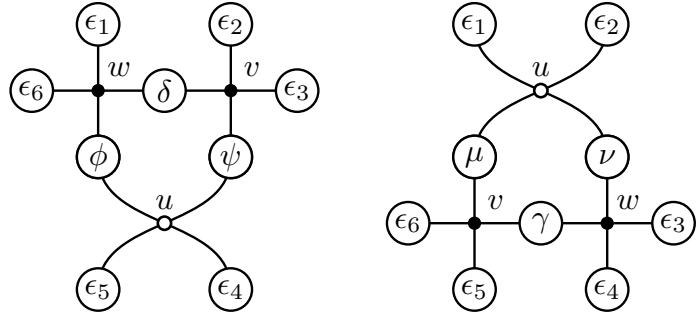
\[ \includegraphics{ybrrot.mps} \hspace{2em}
     \includegraphics{yblrot.mps} \]
\end{theorem}

\begin{proof}
  This may be deduced from Theorem 3 of {\cite{hkice}} with a little work. By
  rotating the diagram $90^{\circ}$ and changing the signs of the horizontal
  edge spins, this becomes a case of that result. The R-matrix is not the matrix
  for $u$ given above, but a constant multiple. We leave the details to the
  reader.
\end{proof}

Let $\sigma_i$ be the map on sequences $\alpha = (\alpha_1, \alpha_2, \cdots)$
that interchanges $\alpha_i$ and $\alpha_{i + 1}$. Let $\lambda$ be a
partition. We will show that sometimes:
\begin{equation}
  \label{fsfsym} s_{\lambda} (z|\alpha) = s_{\lambda} (z| \sigma_i \alpha)
\end{equation}
and sometimes
\begin{equation}
  \label{schurdecomp} s_{\lambda} (z|\alpha) = s_{\lambda} (z| \sigma_i \alpha) +
  s_{\mu} (z| \sigma_i \alpha) (\alpha_{i} - \alpha_{i+1}),
\end{equation}
where $\mu$ is another partition. The
next Proposition gives a precise statement distinguishing between the two
cases (\ref{fsfsym}) and (\ref{schurdecomp}).

\begin{proposition}
  (i) Suppose that $i + 1 \in \lambda + \rho$ but that $i \notin \lambda +
  \rho$. Let $\mu$ be the partition characterized by the condition that $\mu +
  \rho$ is obtained from $\lambda + \rho$ by replacing the unique entry equal
  to $i + 1$ by $i$. Then (\ref{schurdecomp}) holds with this $\mu$.
  
  (ii) If either $i + 1 \notin \lambda + \rho$ or $i \in \lambda + \rho$ then
  (\ref{fsfsym}) holds.
\end{proposition}

\medbreak\noindent
For example, suppose that $\lambda = (3, 1)$ and $n = 5$. Then $\lambda + \rho
= (8, 5, 3, 2, 1)$. If $i = 4$, then $i + 1 = 5 \in \lambda + \rho$ but $i
\notin \lambda + \rho$, so (\ref{schurdecomp}) holds with $\mu + \rho = (8, 4,
3, 2, 1)$, and so $\mu = (3)$.

\medbreak
\begin{proof}
  We will use Theorem~\ref{vertybe} with $\alpha = \alpha_i$ and $\beta =
  \alpha_{i + 1}$. The parameter $t$ may be arbitrary for the following
  argument. We first take $i = n$ and attach the vertex $u$ below the $i$ and
  $i + 1$ columns, arriving at a configuration like this one in the case $n =
  3,$ $\lambda = (4, 3, 1)$.
  
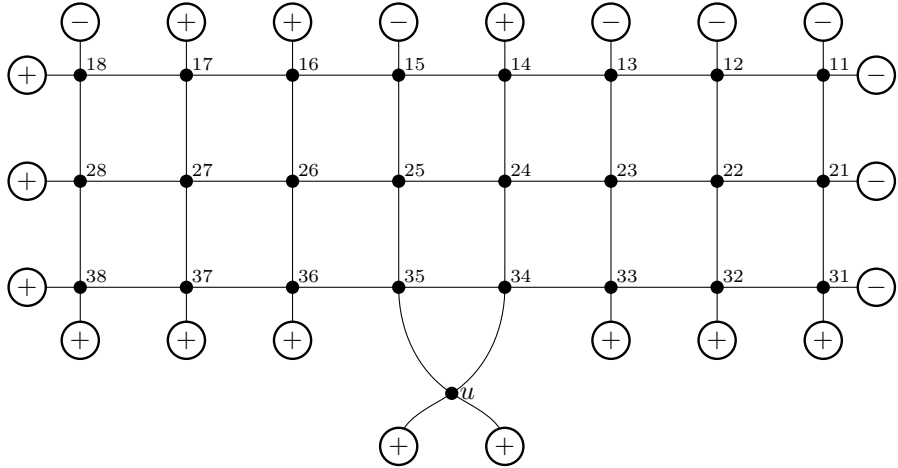
\[ \includegraphics{lattice4.mps} \]
  There is only one legal configuration for the spins of the edges between $u$
  and the two edges above it, which in the example connect with $(3, 5)$ and
  $(3, 4)$: these must both be $+$. The Boltzmann weight at $u$ in this
  configuration is unchanged, and the partition function of this system equals
  that of $\mathfrak{S}^{\Gamma}_{\lambda, t}$. After applying the Yang-Baxter
  equation, we arrive at a configuration with the $u$ vertex above the top
  row, as follows:
  
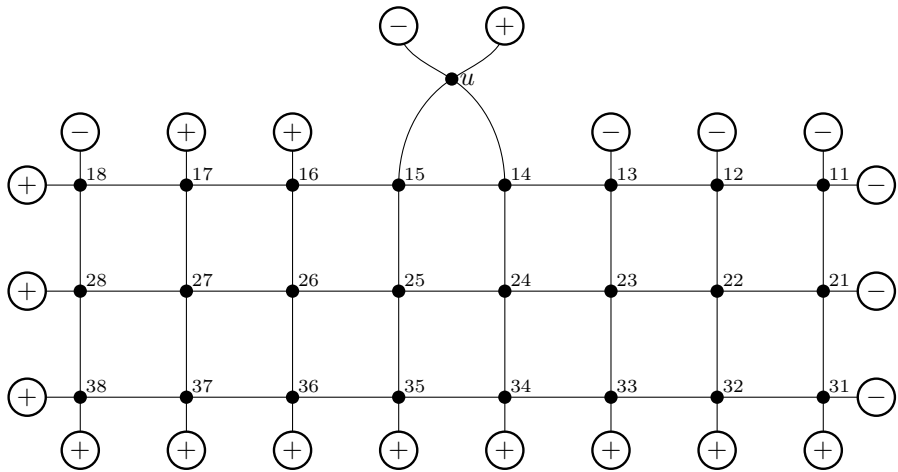
\[ \includegraphics{lattice5.mps} \]
  Now if we are in case (i), the spins of the two edges are $-$, $+$ then
  there are two legal configurations for the vertex, and separating the
  contribution these we obtain
  \[ Z (\mathfrak{S}_{\lambda, t}) = \left( \sigma_i Z (\mathfrak{S}_{\lambda,
     t}) \right) + (\alpha-\beta) \left( \sigma_i Z (\mathfrak{S}_{\mu, t})
     \right) . \]
  If we are in case (ii), there is only one legal configuration, so $Z
  (\mathfrak{S}_{\lambda, t}) = \left( \sigma_i Z (\mathfrak{S}_{\lambda, t})
  \right) .$
\end{proof}

\begin{corollary}
  Let $\lambda$ be a partition, and let $l$ be the length of $\lambda$. If $n
  \geqslant l + i$ then $s_{\lambda} (z|a) = s_{\lambda} (z| \sigma_i a)$.
\end{corollary}

\begin{proof}
  For if $l$ is the length of the partition $\lambda$, then the top edge spin in
  column $j$ is $-$ when $j \leqslant n - l$, and therefore we are in
  case~(i).
\end{proof}

This Corollary implies Macdonald's observation that the factorial Schur
functions are asymptotically symmetric in the $\alpha_i$.

\section{The Dual Cauchy Identity}

Let $m$ and $n$ be positive integers. For a partition $\lambda=(\lambda_1,\ldots,\lambda_n)$ with $\lambda_1\leq m$, we define a new partition $\hat\lambda=(\hat\lambda_1,\ldots,\hat\lambda_m)$ by
\[
 \hat\lambda_i=|\{j\mid \lambda_j \leq m-i\}|.
\]
After a reflection, it is possible to fit the Young diagrams of $\lambda$ and $\hat\lambda$ into a rectangle.

We shall prove the following identity, known as the dual Cauchy identity. Another
proof may be found in Macdonald~\cite{MacdonaldVariations} (6.17). In view of
the relationship between factorial Schur functions and Schubert polynomials,
this is equivalent to a statement on page~161 of Lascoux~\cite{LascouxBook}.
See also Corollary 2.4.8 of Manivel~\cite{Manivel} for another version of
the Cauchy identity for Schubert polynomials.

\begin{theorem}[Dual Cauchy Identity]
 For two finite alphabets of variables $x=(x_1,\ldots,x_n)$ and $y=(y_1,\ldots,y_m)$, we have
\[
 \prod_{i=1}^n \prod_{j=1}^m (x_i+y_j) = \sum_\lambda s_\lambda(x|\alpha) s_{\hat\lambda}(y|-\alpha).
\]
The sum is over all partitions $\lambda$ with at most $n$ parts and with $\lambda_1\leq m$.
\end{theorem}
\begin{proof}
The proof will consist of computing the partition function of a particular six-vertex model in two different ways.
We will use the weights $v_\Gamma(i,j,t)$ introduced in Section
\ref{toksec} with $t=1$ and for our
parameters $(z_1,\ldots,z_{m+n})$, we will take the sequence $(y_m,\ldots,y_1,x_1,\ldots,x_n)$.

As for the size of the six-vertex model we shall use and the boundary
conditions, we take a $(m+n)\times (m+n)$ square array, with positive spins on
the left and lower edges, and negative spins on the upper and right edges. By
Theorem \ref{icerepn} (with $\lambda=0$), the partition function of this array
is
\[
 \prod_{i< j}(y_i+y_j) \prod_{i,j}(y_i+x_j) \prod_{i<j}(x_i+x_j).
\]

We shall partition the set of all states according to the set of spins that
occur between the rows with parameters $y_1$ and $x_1$. Such an
arrangement of spins corresponds to a partition $\lambda$ in the usual way,
i.e. the negative spins are in the columns labeled $\lambda_i+n-i+1$. In this
manner we can write our partition function as a sum
\[
 \sum_\lambda Z_\lambda^{\text{top}} Z_\lambda^{\text{bottom}}.
\]

Here $Z_\lambda^{\text{top}}$ is the partition function of the system with $m$
rows, $m+n$ columns, parameters $y_m,\ldots,y_1$ and boundary
conditions of positive spins on the left, negative spins on the top and right,
and the $\lambda$ boundary condition spins on the bottom. And
$Z_\lambda^{\text{bottom}}$ is the partition function of the system with $n$
rows, $m+n$ columns, spectral parameters $x_1,\ldots,x_n$ and the usual
boundary conditions for the partition $\lambda$.

By Theorem \ref{icerepn}, we have 
\[Z_\lambda^{\text{bottom}}=\prod_{i<j}(x_i+x_j) s_\lambda(x|\alpha).\]

It remains to identify $Z_\lambda^{\text{top}}$. To do this, we perform the following operation to the top part of our system. We flip all spins that lie on a vertical strand, and then reflect the system about a horizontal axis. As a consequence, we have changed the six-vertex system that produces the partition function $Z_\lambda^{\text{top}}$ into a system with more familiar boundary conditions, namely with positive spins along the left and bottom, negative spins along the right hand side, and along the top row we have negative spins in columns $\hat\lambda_i+m-i+1$.
The Boltzmann weights for this transformed system are (since $t=1$)
\[ \begin{array}{|c|c|c|c|c|c|c|}
       \hline
       \begin{array}{c}
         \tmop{Gamma}\\
         \text{Ice}
       \end{array} & \includegraphics{gamma1a.mps} &
       \includegraphics{gamma6a.mps} & \includegraphics{gamma4a.mps} &
       \includegraphics{gamma5a.mps} & \includegraphics{gamma2a.mps} &
       \includegraphics{gamma3a.mps}\\
       \hline
       \text{\begin{tabular}{c}
         Boltzmann\\
         weight 
       \end{tabular}} & 1 & y_i+\alpha_j & 1 & y_i-\alpha_j & 2y_i & 1\\
       \hline
     \end{array} \]
Now again we use Theorem \ref{icerepn} to conclude that
\[
Z_\lambda^{\text{top}}=\prod_{i<j}(y_i+y_j) s_{\hat\lambda}(y|-\alpha).
\]
Comparing the two expressions for the partition function completes the proof.\end{proof}

\hbox to 1in{\hrulefill}

\medbreak\noindent
{\footnotesize
Bump: Department of Mathematics, Stanford University, Stanford,
CA 94305-2125 USA. Email: \texttt{bump@math.stanford.edu}

\medbreak\noindent
McNamara: School of Mathematics and Statistics, University of Sydney, NSW,
Australia. Email: \texttt{mcnamara@maths.usyd.edu.au}

\medbreak\noindent
Nakasuji: Department of Information and Communication Sciences,
Faculty of Science and Technology, Sophia University, 7-1 Kioi-cho,
Chiyoda-ku, Tokyo 102-8554, Japan. Email: \texttt{nakasuji@sophia.ac.jp}
}

\end{document}